\documentclass[a4paper, reqno,  envcountsame]{amsart}
\usepackage[usenames,dvipsnames]{color}
\usepackage{tikz-cd}
\usepackage{thmtools}

\usepackage{amsrefs}

 \usepackage[colorlinks,citecolor=blue,urlcolor=blue, linkcolor=blue]{hyperref}

 \usepackage{amsmath, amssymb, xspace}
\usepackage[capitalise]{cleveref}
 \usepackage[all]{xy}
  
\usepackage{amsmath,amssymb,amscd,amsthm,amsfonts,amstext,amsbsy,mathrsfs,upgreek,mathtools,stmaryrd,enumitem,bbm}

\usetikzlibrary{arrows,snakes,positioning,backgrounds,shadows}
\newtheorem{theorem}{Theorem}[section]

\newtheorem{fact}[theorem]{Fact}

\newtheorem{prop}[theorem]{Proposition}

\newtheorem{convention}[theorem]{Convention}

\newtheorem{lemma}[theorem]{Lemma}

\newtheorem{claim}[theorem]{Claim}
\newtheorem*{claim*}{Claim}
\theoremstyle{definition}
\newtheorem{definition}[theorem]{Definition}

\newtheorem{remark}[theorem]{Remark}

\newcommand{\mrm}[1]{\mathrm{#1}}

\newcommand{\Sym}{\mrm{Sym}}

\newcommand{\NN}{{\mathbb{N}}}

\newcommand{\sub}{\subseteq}

\newcommand{\bi}{\begin{itemize}}
\newcommand{\ei}{\end{itemize}}
\newcommand{\bc}{\begin{center}}
\newcommand{\ec}{\end{center}}

\newcommand{\ES}{\emptyset}

\newcommand{\la}{\langle}
\newcommand{\ra}{\rangle}

\newcommand{\n}{\noindent}

\newcommand{\sss}{\sigma}

\newcommand{\lland}{\, \land \, }

\newcommand\+[1]{\mathcal{#1}}

\newcommand{\ol}{\overline}

\newcommand{\lra}{\leftrightarrow}
\newcommand{\LR}{\Leftrightarrow}

\DeclareMathOperator \Th{Th}
\DeclareMathOperator \Gr{\mathbf{G}}

  \DeclareMathOperator{\Aut}{Aut}
    \DeclareMathOperator{\Inn}{Inn}
        \DeclareMathOperator{\Int}{Int}
      \DeclareMathOperator{\Out}{Out}

\renewcommand{\S}{S_\infty}
\renewcommand{\hat}{\widehat}

\newcommand{\G}{\mathcal{G}}

\newcommand{\LC}{\mathrm{LC}} 
\newcommand{\RC}{\mathrm{RC}}

\newcommand{\idempotent}{{idempotent}} 
\newcommand{\idempotents}{{idempotents}}

\newcommand{\id}{\mathrm{id}}

\newenvironment{enumerate-(a)}{\begin{enumerate}[label={\upshape (\alph*)}, leftmargin=2pc]}{\end{enumerate}}
\newenvironment{enumerate-(a)-r}{\begin{enumerate}[label={\upshape (\alph*)}, leftmargin=2pc,resume]}{\end{enumerate}}
\newenvironment{enumerate-(A)}{\begin{enumerate}[label={\upshape (\Alph*)}, leftmargin=2pc]}{\end{enumerate}}
\newenvironment{enumerate-(A)-r}{\begin{enumerate}[label={\upshape (\Alph*)}, leftmargin=2pc,resume]}{\end{enumerate}}
\newenvironment{enumerate-(i)}{\begin{enumerate}[label={\upshape (\roman*)}, leftmargin=2pc]}{\end{enumerate}}
\newenvironment{enumerate-(i)-r}{\begin{enumerate}[label={\upshape (\roman*)}, leftmargin=2pc,resume]}{\end{enumerate}}
\newenvironment{enumerate-(I)}{\begin{enumerate}[label={\upshape (\Roman*)}, leftmargin=2pc]}{\end{enumerate}}
\newenvironment{enumerate-(I)-r}{\begin{enumerate}[label={\upshape (\Roman*)}, leftmargin=2pc,resume]}{\end{enumerate}}
\newenvironment{enumerate-(1)}{\begin{enumerate}[label={\upshape (\arabic*)}, leftmargin=2pc]}{\end{enumerate}}
\newenvironment{enumerate-(1)-r}{\begin{enumerate}[label={\upshape (\arabic*)}, leftmargin=2pc,resume]}{\end{enumerate}}

\begin{document}

 \title {Automorphism groups of  non-Archimedean  groups}
 
\author{Andr\'e Nies}
\author{Philipp Schlicht}

\thanks{The first and second author were partially supported by the Marsden fund of New Zealand,   19-UOA-1931. 
The second listed author gratefully acknowledges the support of INdAM-GNSAGA. 
}

\noindent 
\address{A.\  Nies, School of Computer Science,  The University of Auckland, Private Bag 92019, Auckland 1142, New Zealand. \newline  \texttt{andre@cs.auckland.ac.nz}}

\address{P.\ Schlicht, Dipartimento di ingegneria dell'informazione e scienze matematiche, Università di Siena, via Roma 56, 53100 Siena, Italia. \newline  \texttt{philipp.schlicht@unisi.it}}

\maketitle

\begin{abstract} 
Let $\Aut(G)$ denote the group of (bi-)continuous automorphisms 
of a non-Archimedean  Polish group~$G$. 
We show that for any such $G$ with an invariant countable basis of open subgroups, the group  $\Aut(G)$ carries a unique Polish topology that makes its natural action on $G$ continuous. 
Furthermore, for any class of groups allowing a Borel assignment of such bases, there is a functorial duality to a class of countable groupoids with a meet operation,  extending 
 work of the authors with Tent  (Coarse groups, and the isomorphism problem for oligomorphic
 groups, Journal of Mathematical Logic, 2021).
This provides an alternative description of the topology of $\Aut(G)$. 
The results hold for instance for the class of locally Roelcke precompact non-Archimedean groups, which contains most classes studied previously. 
We further provide a model-theoretic proof that the outer automorphism group $\Out(G)$ of an oligomorphic group $G$ is locally compact, a result due to Paolini and the first author (arXiv:2410.02248). 
\end{abstract} 

\tableofcontents

\section{Introduction}

A Polish group $G$
is called \emph{non-Archimedean  (nA)}  if $G$ has a countable basis  $\+ S_G = \{ U_n \colon n \in \omega\}$ of neighbourhoods  of the identity that consists of open subgroups.   
Such a  group $G$  is topologically isomorphic to a closed subgroup of $\S$, the group of permutations of $\NN$ with the topology of pointwise convergence. (To~see this, one  lets $G$ act from the left on the  set of left cosets of subgroups in $\+ S_G$.)  Conversely, each closed subgroup of $\S$ is nA, via taking as  $U_n$ the permutations in $G$ that fix each $i< n$. One verifies that the closed subgroups of $\S$ are precisely  the automorphism groups of structures with domain $\omega$.  The class of nA groups enjoys several permanence  properties: closed subgroups, quotients by closed normal subgroups, and countable Cartesian products of nA groups are again nA. 

Our paper  focusses on   two groups derived from a nA group $G$:  $\Aut(G)$,  the group of  \textit{continuous} automorphisms of~$G$, and   $\Out(G)$, the  quotient of $\Aut(G)$ by  its  normal subgroup of inner automorphisms.     One says that a closed subgroup $G$ of $\S$ is    \textit{oligomorphic} if for each~$n\in \omega$, its action on $\omega^n$ has only finitely many orbits. We will in particular address  $\Aut(G)$ and $\Out(G)$ for oligomorphic $G$. 

The paper   has three interrelated parts. 

\begin{enumerate} 
	\item[(A)] The natural   action of $\Aut(G)$ on $G$ is given by $\alpha \cdot g  = \alpha(g)$ for $\alpha \in \Aut(G)$ and $g \in G$. We provide  a sufficient criterion    when $\Aut(G)$   has a compatible Polish  topology that makes this  action   continuous:  \textit{$G$  has a countable neighbourhood basis   $\+ S_G$ of the neutral element   consisting  of open subgroups   that can be chosen invariant under this action}. In this case, $\Aut(G)$ is   nA itself. 
	\item[(B)] The   class of oligomorphic groups satisfies the criterion; also, $\Inn(G)$ is closed in $\Aut(G)$. So $\Out(G)$ is a Polish group. We study it  through the model-theoretic notion of bi-interpretations, and use this to give  a model-theoretic  proof of the result in~\cite{Nies.Paolini:24} that $\Out(G)$ is totally disconnected and  locally compact. 
	\item[(C)]  Given  a Borel class $\mathbf G$ of nA groups that satisfies  a uniform version   of the criterion in (A) due to Kechris et al.\ \cite{Kechris.Nies.etal:18}, we provide a Borel equivalence of the category $\mathbf G$ with isomorphism, and a category of countable structures with domain the coset of subgroups in $\+ S_G$. 
	This provides an alternative way to  obtain the    Polish topology on $\Aut(G)$ making  its action on $G$ continuous. 
\end{enumerate}    

 The study of  $\Out(G)$ is motivated in part  by the question whether the isomorphism relation between  oligomorphic groups is smooth in the sense of Borel reducibility. 
An upper bound on this relation is known by~\cite{Nies.Schlicht.etal:21}: it is essentially countable. However the precise complexity, first asked in \cite{Kechris.Nies.etal:18}, remains unknown.      Certain subclasses are known to be smooth~\cite{Nies.Paolini:24}, such as the automorphism groups of $\omega$-cateogrical structures    without algebraicity. 

In contrast,    the conjugacy relation on the  Borel  space of oligomorphic groups is smooth~\cite{Nies.Schlicht.etal:21}. 
Towards answering the question, it is thus useful to know whether an isomorphism between two permutation groups $G$ and $H$ is induced by a conjugation with a permutation of $\NN$. 
For a single permutation group $G$, $\Out(G)$ is trivial iff  every continuous automorphism of $G$ is given  by conjugating with a permutation of $\NN$ in $G$.

We discuss the   parts (A)-(C)   in some detail, with proofs delegated to  the main body of the paper. First we state the    definition of Roelcke precompact Polish groups   in  the nA case.
\begin{definition} \label{def:RP}
	A non-Archimedean group $G$ is \textit{Roelcke precompact} if each open subgroup has only finitely many double cosets. 
\end{definition}
\subsection*{(A) When  is $\Aut(G)$   non-Archimedean?}   
 
  The main  result  for this part will be   summarised here; the detailed version is  \cref{th:1}.

 \begin{theorem}   \label{th:1sum}
 Suppose a nA group satisfies the criterion in ({\rm A}) above.  
Then there is a copy $\hat G$ of    $G$ as a closed subgroup of $\S$ such that  the group $\Aut(G)$ is     isomorphic   to the Polish group that is given as the normaliser of $\hat G$ in $\S$ by its centraliser in $\S$.   The induced topology on $\Aut(G)$ is the unique Polish topology that makes the action on $G$ continuous. 
 \end{theorem} 
The criterion implies that  $\Inn(G)$ is a Borel subgroup of $\Aut(G)$ that is Polishable (\cref{prop:Polishable}).
As an application of the criterion, $\Aut(G)$  is itself  nA for several natural classes of nA groups. This includes the  locally Roelcke precompact groups, where the neighbourhood basis in (A) consists of the Roelcke precompact open subgroups of~\cref{def:RP}. We note that if $G$ is Roelcke precompact, then $\Inn(G)$ is in fact a closed   subgroup  of $\Aut(G)$ \cite[Th.\  2.7]{Nies.Paolini:24}. 


\subsection*{(B) Oligomorphic groups and bi-interpretations} 

Each oligomorphic group is  Roelcke precompact by  \cite[Th.\ 2.4]{Tsankov:12}.  If $G$ is oligomorphic,   the outer automorphism group $\Out(G)= \Aut(G)/\Inn(G)$ is t.d.l.c.\ by~\cite[Th.\ 3.10]{Nies.Paolini:24}. We wish to describe $\Out(G)$ based on  the theory of any $\omega$-categorical structure that has   $G$ as its automorphism group. 
For an $\omega$-categorical theory $T$, let $\+ B(T)$ be the group of self-interpretations of $T$ that have an inverse, all up to definable bijections between sorts of $T^\mathrm{{eq}}$  (for formal detail see \cref{def: BT}).  The following reproves the result   that $\Out(G)$ is t.d.l.c.\ using model theory, and gives a model theoretic  description of $\Out(G)$ based only on the theory of the underlying structure. 
\begin{theorem} \label{th:BI Out}
	Let $G$ be an  oligomorphic group, and let $T$ be the elementary theory of a    structure $M$ such that    $G= \Aut(M)$.  
\begin{itemize} \item[(i)]  The group $\+ B(T)$ is totally disconnected, locally compact. \item[(ii)] $\+ B(T)$ is    topologically isomorphic to $\Out(G)$. \end{itemize}
\end{theorem}

\begin{remark} 
	Outer automorphism groups were introduced to  study  finite groups.   
	For countable groups, outer automorphism groups are important in the study of mapping class groups of surfaces.  Dehn, Nielsen and Baer showed that the extended mapping class group $\mathrm{MCG}^+(S)$ of a compact closed orientable surface $S$ of genus $g\geq 1$ is isomorphic to $\Out(\pi(S))$. Here  $\mathrm{MCG}^+(S)$ denotes the quotient of the  group of  orientation preserving homeomorphisms of $S$ by the connected component of the identity,  and $\pi(S)$ is the fundamental group \cite[Theorem 8.1]{Farb.Margalit:11}. 
	\end{remark}

\subsection*{(C) Borel equivalence of  categories $\Gr$ and $\+ M$} The closed subgroups of $\S$ form a standard Borel space $\+ U(\S)$, which is a subspace of the usual Effros space  $\+ F(\S)$ of closed subsets of $\S$.  Kechris et al.\ \cite{Kechris.Nies.etal:18}  studied Borel subclasses  $\Gr$   of $\+ U(\S)$  of groups $G$ for which the assignment of $\+ S_G$   to $G$  is Borel and   isomorphism  invariant. To   assign $\+ S_G$ to $G$ in a Borel way means that the relation \bc $\{(G,U)\colon G\in \mathbf {Gr} \land U \in \+N_G\}$ \ec is Borel.  The invariance condition means that  if $f\colon G \to H$ is topological  isomorphism then $U \in \+N_G \lra f(U) \in \+ S_H$.  In particular, $\+ S_G$ is closed under the action of $\Aut(G)$. Kechris et al.\ \cite{Kechris.Nies.etal:18} showed that the  topological isomorphism relation on such a class $\Gr$  is classifiable by countable structures.   Equivalence of categories was introduced by  Eilenberg and MacLane, and will be  recalled in \cref{def: equivalence categories}. We  establish an equivalence of
\begin{itemize}
	\item  [(1)] the category   that has as  objects the groups in such a class, and as morphisms their   topological isomorphisms, 
\item  [(2)] the category that has as  obects a certain Borel  class of countable  structures for a finite signature, and  as morphisms their  isomorphisms.   \end{itemize}
 The   functors needed to establish this equivalence will be   Borel.

\section{Non-Archimedean topology on $\Aut(G)$}

Let $H$ be a closed subgroup of $\S$ (denoted $H \le_c \S$).  By $N_{\S}(H)$ we denote the normaliser,  and by $C_{\S}(H)$ the centraliser,  of $H$ in $\S$. Note that both are closed subgroups of $\S$.  
\begin{theorem}[Full version of \cref{th:1sum}] \label{th:1} {\rm 
	Suppose $G$ is a Polish 
	 group with a countable neighbourhood  basis $\+ S_G$  of the neutral element   consisting of open subgroups, such that $\+ S_G$    is   invariant under the action of $\Aut(G) $ on the open subgroups. 
	
	\begin{enumerate}\item[(i)]  There is a group $\hat G \le_c \S$ and a topological isomorphism $\Theta \colon G \to \hat G$ such that $$N_{\S}(\hat G) / C_{\S}(\hat G)\cong \Aut(G) $$  via sending an $\alpha \in N_{\S}(\hat G)$ to its conjugation action on $G$. Using this   isomorphism, $\Aut(G)$   can be topologised as a    non-Archimedean group  in such a way that its  action on $G$ is continuous.
		\item[(ii)]   A  neighbourhood basis of the identity for this topology on $\Aut(G)$  is given by   the 
		subgroups   of the form \begin{equation} \label{eqn: basis ident} \{\Phi  \in \Aut(G)\colon\bigwedge_{i =1}^n   \Phi(A_i)= A_i\}, \tag{$*$}\end{equation}  where $A_1, \ldots, A_n $ are cosets of subgroups  in $ \+ S_G$.    
		
		\item[(iii)] This  topology on $\Aut(G)$ is the unique Polish topology  that makes    the   action  of $\Aut(G)$ on $G$   continuous.
	\end{enumerate}  }
\end{theorem}

 The proof uses  the basic idea  from the proof of Kechris et al.\ \cite[Th.\ 3.1]{Kechris.Nies.etal:18},     removing   references to Borelness and uniformity. We   use  $\+ S_G$ for the neighbourhood basis of $1$,  instead of the notation $\+ N_G$   there.  
\begin{proof}[Proof of  \cref{th:1}(i)]  Let   $\+ S_G^*$ denote the set of left  cosets of the subgroups in $\+ S_G$. Then $G$ acts from the left on $\+ S_G^*$. Since $\+ S_G^*$ is countably infinite, we can fix a bijection $\rho_G \colon \omega \to  \+ S_G^*$, so  the  left  action of an element  $g$ on $\+ S_G^*$ corresponds    to a permutation $\Theta(g) \in \S$.    We let $\hat G$ denote  subgroup of $\S$ that is the range of  $ \Theta$ with the  topology inherited from $\S$. 
\begin{claim}[\cite{Becker.Kechris:96}*{Th.\ 1.5.1}; see also \cite{Kechris.Nies.etal:18}*{Claim 3.2}] \label{cl:wow} \  \newline The map $\Theta \colon G \to \hat G $ is a topological group isomorphism.  In particular, since $G$ is Polish, the group  $\hat G$ is a  closed subgroup of $\S$.  \end{claim}

	For $\alpha \in N_{\S}(\hat G)$ let $K_\alpha\in \Aut(\hat G)$ denote conjugation by $\alpha$, namely $K_\alpha(h)= \alpha \circ h \circ \alpha^{-1}$.
\begin{lemma}  A   retraction $\Gamma \colon N_{\S}(\hat G) \to  \Aut(G)$   with kernel  $C_{\S}(\hat G)$ is given by 	\[\Gamma(\alpha)= \Theta^{-1} \circ K_\alpha \circ \Theta\] 

\end{lemma}	

\begin{proof}[Proof of Lemma]
 Clearly $\Gamma $ is a group homomorphism. To show that $\Gamma$ is a retraction, we will define a  homomorphism  $\Delta \colon \Aut(G) \to \S$ such that  $\Gamma \circ \Delta$ is the identity on $\Aut(G)$.  By our hypothesis that $\+ S_G$ is    invariant under the action of $\Aut(G)$, any $\phi  \in \Aut(G)$  induces a bijection $\bar \phi  \colon \+ S_G^* \to \+ S_G^*$ via \bc $\bar \phi(rU) = \phi(r) \phi(U)$ for $U \in \+ S_G$ and $r \in G$. \ec  Let   $\Delta(\phi)$ be this bijection viewed as a permutation of $\omega$, that is,  
	$\Delta(\phi) =  \rho_G^{-1} \circ \bar \phi  \circ \rho_G $.

\begin{claim}  Given $\phi \in \Aut(G)$ let 
	$\alpha:= \Delta(\phi) $. For any $g \in G$ we have 
\[	\alpha\circ \Theta(g) \circ \alpha^{-1} = \Theta(\phi(g)).\]  Thus $\alpha \in N_{\S}(\hat G)$       and $\Gamma(\alpha) = \phi$. \end{claim}
\n 	To see this, let $A \in \+ S_G^*$ and write $A  =\phi(r)\phi(U)$ for some $r\in G$, and $U \in \+ S_G$. 
Then (suppressing $\rho_G$) we have  $\alpha^{-1}(A)= rU$, so 
$$( \alpha\circ \Theta(g) \circ \alpha^{-1} ) (A) = \alpha(grU)= \phi(g) \phi(r)\phi(U) = \phi(g) A,$$
\n as required. \end{proof}



   The group $N_{\S}(\hat G)/C_{\S}(\hat G)$ is algebraically isomorphic to $\Aut(G)$ via the isomorphism induced by the retraction $\Gamma$ that sends $\alpha$ to its conjugation action on $\hat G \cong G$. It   is non-Archimedean as a quotient of $N_{\S}(\hat G)$ by a closed normal subgroup.   We use this isomorphism to transfer its  topology to $\Aut(G)$. This makes the action on $G$ continuous,  because the action of $N_{\S}(\hat G)$ on $\hat G$ by conjuguation  is continuous, and for each $g \in G$ and $\alpha \in N$, we have $\Theta^{-1}(K_\alpha(\Theta(g)))= \Gamma(\alpha)(g)$.   \end{proof}

\begin{remark}
	Since $\Delta \circ \Gamma$ is continuous,   $\Delta$   is   a topological isomorphism    between  $\Aut(G)$ and a closed subgroup $H$ of $\S$ contained in   $N_{\S}(\hat G)$. Thus $N_{\S}(\hat G)$ is a topological split extension:   $H \cap C_{\S}(\hat G)= \{1\}$ and $H C_{\S}(\hat G) = N_{\S}(\hat G)$.  
\end{remark}

\begin{proof}[Proof of  \cref{th:1}(ii)]
	Write $C=  C_{\S}(\hat G)$.   Given a subgroup $\+ A$ as in~(\ref{eqn: basis ident}), let  $A_i$ be a left coset of a subgroup  $U_i \in \+ S_G$. Let $\+ U$ be the pointwise stabiliser of the set $\{ U_1, A_1, \ldots, U_n, A_n\}$ where we identify the elements of $\+ S_G^*$ with their code numbers. 
	  
	  We  claim that $\alpha\in C \+ U  $ iff $\Gamma(\alpha)\in \+ A$. This   suffices for (ii) since the subgroups of the  form $C\+ U /C$ form a neighbourhood basis  of the identity for  $N_{\S}(\hat G)/C$.
	  
	    If $\alpha \in \+ U$ then for each $i$, and  for each $g\in G$, writing $p = \Theta(g)$,   we have
	\[ g\in A_i \lra p(U_i)= A_i \lra \alpha \circ p \circ \alpha^{-1}(U_i)= A_i \lra K_\alpha(p)(U_i) = A_i\lra \Gamma(\alpha)(g) \in A_i.\]
	Thus $\Gamma(\alpha)\in \+ A$.
	
	Now suppose $\phi= \Gamma(\alpha)\in \+ A$. Since $U_i= A_i^{-1} A_i$, this implies that the automorphism $\phi$   fixes $U_i$. Hence  $\Delta(\phi)\in \+ U$. Since $\alpha (\Delta(\phi))^{-1} \in C$, this verifies the claim.
\end{proof}
%
%
%
%
%
%
%

We next establish \cref{th:1}(iii),  the uniqueness of a  Polish topology on $\Aut(G)$ that makes the action on $G$ continuous. We prove a  more general result. To obtain \cref{th:1}(iii) from it, one chooses as the basis $\+ B $ the set of  left cosets of  subgroups in  $\+ S_G$.

\begin{prop} \label{prop:coarsest} {\rm  Let $G$ be a Polish group, and suppose that    there is a countable basis $\+ B$ of clopen sets for the topology of $G$  such that  $\+ B$ is invariant under the  natural action of $\Aut(G)$.    Let $\tau$ be   the  topology induced on $\Aut(G)$ by the neighbourhood basis ($*$) of the identity, for $A_1, \ldots, A_n \in \+ B$.
	\bi \item[(a)] If  $\sss$ is a  Baire group topology on  $\Aut(G)$ making its action on $G$ continuous, then $\tau \sub \sss$. 
	\item[(b)]  $\tau$ is the unique Polish group topology on $\Aut(G)$ making its  action on $G$ continuous.   \ei}\end{prop}

\begin{proof} (a)   Every  Baire measurable homomorphism from a Baire group $K$ to a separable group $L$  is continuous~\cite[Th.\ 9.9.10]{Kechris:95}.  (Here, to be  Baire measurable means that the pre-image of every open set  in $L$ has  meager  symmetric difference with some  open set in $K$.) We will show that  the identity homomorphism $(\Aut(G), \sss) \to(\Aut(G), \tau)$ is   Baire measurable; we  conclude that  it is continuous and hence $\tau \sub \sss$. 
	
	Each Borel set for $\sss$ has a meager symmetric difference with some open set in the sense of $\sss$ (i.e., it has the property of Baire for $\sss$).  We  verify  that   the $\tau$-open sets in the neighbourhood  basis ($*$) of the identity automorphism are $G_\delta$ for $\sss$: this  will imply that   each  $\tau$-open set is Borel for~$\sss$.   
	
	To do so, for $A\in \+ B$, we verify that  the  set 
	\bc $\+ S=\{ \Phi\in \Aut(G) \colon \, \Phi(A) = A\} $  \ec is $G_\delta$ with respect to  $\sss$. Using that $A$ is closed, let $D$ be a countable dense subset of $A$. Then \bc $\+ S= \bigcap_{g \in D} \{\Phi \in \Aut(G)  \colon \ \Phi(g) \in A \lland \Phi^{-1}(g) \in A\}$.   \ec Thus $\+ S$  is $G_\delta$ for $\sss$ because $A$ is open and $\sss$ makes the action of $\Aut(G)$ on $G$ continuous.  Now each set in ($*$) is a finite intersection of such sets, and hence $G_\delta$ for $\sss$ as well.  
	
	\n (b) If $\sss$ is Polish, then it is Baire, so $\tau \sub \sss$. This implies   $\tau = \sss$ because no Polish group topology on a group can be properly contained in another (see, e.g., \cite[2.3.4]{Gao:09}). 
\end{proof}

We close this section with some remarks on the group of inner automorphisms of $G$  as a subgroup of  $\Aut(G$).  Note that if 
 $\alpha = \Theta(g)$ then $K_\alpha$ defined after \cref{cl:wow} induces the inner automorphism of $G$ given by conjugation by $g$:  for each $h \in G$, 
$$\Gamma(\Theta(g))(h)= [\Theta^{-1} \circ K_{\Theta(g)} \circ \Theta](h) = \Theta^{-1}(\Theta(h)^{\Theta(g)})= h^g.$$

The group $\Inn(G)$ is in general not closed in $\Aut(G)$, even for discrete groups $G$ (see \cite[Section 2]{Nies.Paolini:24}).  Via the   isomorphism  $ N_{\S}(\hat G)/ C_{\S}(\hat G) \cong \Aut(G)$ established above,  the subgroup $\Theta(G) C_{\S}(\hat G)/ C_{\S}(\hat G)$ corresponds to $\Inn(G)$,  so this subgroup can fail to be closed. However, it satisfies   a weaker condition. Recall that a Borel subgroup $H$ of a Polish group is Polishable if $H$ carries a (unique) Polish group topology such  that   the $\sss$-algebra is generates coincides with the  Borel sets inherited from $G$.  
\begin{prop} \label{prop:Polishable} Let $G$ be as in \cref{th:1}. Then $\Inn(G)$ is a Borel subgroup of $\Aut(G)$ that is Polishable.
	\end{prop}
	\begin{proof}
		We verify that the  natural map $
		\Phi \colon G / Z(G) \to \Aut(G)$  that sends a  $g\in G$ to its conjugation action on $G$ is continuous. Given a basic open subgroup $\+U$ of $\Aut(G)$  as in  ($*$), suppose $A_i$ is a right coset of $U_i$ and left coset of $V_i$. If $g \in \bigcap_i U_i \cap V_i$ then $A_i^g = A_i$ for each $i$, so $\Phi(g) \in \+ U$.  
		
	Now, since $\Phi$ is injective, by the Lusin-Suslin theorem, $\Phi(A)$ is Borel for each Borel set $A \sub G$. So $\Inn(G)$ is Borel in $\Aut(G)$, and the Borel structure on $G/Z(G)$ equals the  Borel structure on $\Inn(G)$ inherited from $\Aut(G)$. 
	\end{proof}

\begin{remark} \cref{prop:Polishable} implies that  for $G$   as in \cref{th:1}, the outer automorphism group $\Out(G)$ is a group with a Polish cover in the sense of \cite{Bergfalk.Lupini.ea:2024}. 
	\end{remark}

  \section{Oligomorphic groups and invertible self-interpretations}
   
 This section will define the  topological group $\+ B(T)$ of invertible  self-interpretations up to a syntactic notion of homotopy of an $\omega$-categorical theory~$T$ (\cref{def: BT}).  Then  it establishes  in  \cref{th:BI Out} that $\+ B(T)$ is t.d.l.c.\ and isomorphic to $\Out(G)$, the group of  outer automorphisms of~$G$, whenever $G= \Aut(M)$ for some   model $M $ of $ T$ with domain $\omega$.

    Firstly, we  represent  $\Aut(G)$  as a topological group $\+ B(M)$ of  self-interpretations of $M$ (\cref{def:BM}). They  are taken modulo an appropriate equivalence relation denoted $\sim$  that equates them up to  definable bijection between their   domains; granted that, they are  invertible.

    Secondly,  we show that passing from invertible self-interpretations of~$M$ to invertible self-interpretations of its theory    corresponds to passing from $\Aut(G)$ to   $\Out(G)$: by \cref{lem: same Th}, two self-interpretations $\alpha_0, \alpha_1$ of $M$  get identified in this process iff $\pi \circ \alpha_0 \sim \alpha_1$  for some $\pi \in G$, which means for invertible interpretations that the corresponding  automorphisms of $G$ are equal up to an inner automorphism of $G$.

  \subsection{Interpretations between $\omega$-categorical structures}  \begin{convention} \label{conv:MNGH|} Throughout,  let $M ,N$ denote $\omega$-categorical structures, and let  $G = \Aut(M)$ and $H =\Aut(N)$. 
  \end{convention} For interpretations between $\omega$-categorical structures  we ``import" the terminology and notation  of Ahlbrandt and Ziegler~\cite{Ahlbrandt.Ziegler:86}: The  $\omega$-categorical  structures with  interpretations form a category; morphisms are denoted $\alpha \colon M \rightsquigarrow N$. This means that  there is a dimension $k \ge 1$, a definable set  $D \sub M^k$ (always without parameters), and $\alpha $ is a  map $D \to M$ that  is onto with definable kernel $E$; furthermore,   each $N$-definable relation $R$  has an $M$-definable pre-image under $\alpha$. We will abuse notation by  also viewing $\alpha$ as a bijection $D/E \to N$. By $\Int(M,N)$ we   denote the set of morphisms $M \rightsquigarrow N$. 
  
    The structure $M^\mathrm{eq}$ has infinitely many sorts $D/E$ as above. Its  elements are called imaginaries over $M$; see \cite{Tent.Ziegler:12}. A function   $\theta \colon D_0/E_0 \to D_1/E_1$ is called $M$-definable if the relation $\{\la d_0, d_1 \ra \colon 
  \theta(d_0/E_0) = d_1/E_1\}$ is $M$-definable. 
  
   Ahlbrandt and Ziegler  \cite{Ahlbrandt.Ziegler:86}  extend the operation $M \mapsto \Aut(M)$ to  a functor from the category of $\omega$-categorical structures with interpretations to oligomorphic groups with continuous homomorphisms, as follows.
   \begin{definition}[essentially  \cite{Ahlbrandt.Ziegler:86}] \label{def:Autfunctor}
For an interpretation $\alpha \colon M \rightsquigarrow  N $ let    $\Aut(\alpha)(g)= \alpha \circ g^\mathrm{eq} \circ \alpha^{-1}$ for $g \in \Aut(M)$, where  $g^\mathrm{eq}$ is the canonical extension of $g$  to an automorphism of  $M^\mathrm{eq}$.   
   \end{definition} 

    For each sort $D/E$ of $M^\mathrm{eq}$,  the interpretations $\alpha \colon D \to N$ with kernel $E$ can be seen as  a closed subspace of the    space  of functions $D/E \to N$ with the topology of pointwise convergence.  So $\Int(M,N)$  is a topological space   which is a disjoint union of   clopen subspaces corresponding to the  sorts. The $\omega$-categorical  structures with interpretations as morphisms thus form a Polish  category:  each set of morphisms is  a Polish space in a Borel way, and the operations of product and inverse  are   continuous.

    \begin{lemma}  \label{lem: Int} Let $M,N$ and $G,H$ be as in \cref{conv:MNGH|}.
    	
    	\n  The map  $\Int(M,N) \times G \to H$ given by  $\alpha \cdot g = \Aut(\alpha)(g)$ is continuous. 
  \end{lemma}

    	\begin{proof}
    		Suppose $\Aut(\alpha)(g)$ is in the subbasic open set of elements of $H$ that send $r$ to $s$. We have $\alpha (\ol a) = r$ and $\alpha (\ol b) = s$ for some $\ol a, \ol b \in D$ such that $g(\ol a /E) = \ol b /E$. If an interpretation  $\alpha'  \colon D \to N$ with the same  kernel $E$ has the same values as    $\alpha$ on these two tuples and an element  $g' \in G$ also satisfies $g'(\ol a /E) = \ol b /E$,  then   $\Aut(\alpha')(g')$ sends $r$ to $s$ as well.  The set of such pairs $\alpha', g'$ is open. 
    	\end{proof}

    	 	Ahlbrandt and Ziegler	\cite{Ahlbrandt.Ziegler:86} defined  a  ``homotopy"  relation  on $\Int(M,N)$: 
    	\begin{definition}[homotopy of interpretations] \label{def: homotopy}
    For morphisms $\alpha_i \colon M \rightsquigarrow N$ ($i= 0,1$), where  $\alpha_i  \colon D_i/E_i \to N$, one says that $\alpha_0$ is homotopic to $\alpha_1$, written $\alpha_0 \sim \alpha_1$, if there is an $M$-definable bijection   $\theta \colon D_0/E_0 \to D_1/E_1$ such that $\alpha_0= \alpha_1 \circ \theta$.  
    	\end{definition} 
    Note that 	for any      bijection $\theta \colon D_0/E_0 \to D_1/E_1$, the set of pairs $\la \alpha_0, \alpha_1\ra$ in $\Int(M,N)$ with $\alpha_i \colon D_i/E_i \to N$ and $\alpha_0= \alpha_1 \circ \theta$ is closed  in $\Int(M,N)^2$. 
    	\begin{lemma} \label{lem: saturation}
    		(i) the equivalence relation $\sim$ is closed.   (ii) For each open set $U \sub \Int(M,N)$,  the saturation $[U]_\sim$  is also open. 
    	\end{lemma} 
    	\begin{proof} (i)  holds since there are only finitely many $M$-definable bijections $D_0/E_0 \to D_1/E_1$. 
    		
    		\n (ii) 
    		We may assume that $U$ is of the form $\{\alpha \colon \alpha(d_i)= n_i  \ (i = 1, \ldots, k)\}$ where $d_i \in D/E, n_i \in N$. Then $[U]_\sim $ is the union of sets $\{ \alpha \circ \theta \colon \alpha \in U\}$ over all definable bijections $\theta \colon D'/E' \to D/E$; such a  set equals $\{\beta \colon  D'/E' \to N \mid \beta(\theta^{-1} (d_i) )= n_i\}$ and hence  is open. 
    	\end{proof}

    Replacing interpretations by their  equivalence classes with respect to~$\sim$, we obtain  a Polish quotient category.   In particular, $\Int(M,M)/_\sim $ is a Polish monoid.  
    \begin{remark} \label{rem:Pol Monoid} For any Polish  monoid $(S, \cdot,1)$, the group $L$ of (two sided) invertible elements is  a Polish  group as follows.  $L $ is algebraically isomorphic to 
    $	\{ (a, b ) \in S\times \check S \colon ab= ba =1\}$ which is a closed subset of the Polish space $S \times \check S$ also closed under the monoid operation (here $\check S$ is the dual monoid $(S,\cdot', 1)$, where $r \cdot ' s= s \cdot r)$.  The operation of inversion corresponds to exchanging the two components of a pair, which is a    continuous operation. \end{remark}
    
     \subsection{ $\+ B(M) $ is topologically isomorphic to $ \Aut(G)$}

  \begin{definition} \label{def:BM}
	Let $\+ B(M)$ be the Polish  group of invertible elements of the Polish monoid  $\Int(M,M)/_\sim$. 
  \end{definition}
  
   Recall the functor $\Aut$ of  \cref{def:Autfunctor}.  Ahlbrandt and Ziegler  \cite[Thm 1.2]{Ahlbrandt.Ziegler:86} show that a continuous homomorphism $R \colon \Aut(M) \to \Aut(N)$ is of the form $\Aut(\alpha)$ for some $\alpha \colon M \rightsquigarrow N$  if and only   if  the action of $\Aut(M)$ on $N$ induced by  $R$    has only finitely many 1-orbits. They also show  in their Theorem~1.3 that  $\Aut(\alpha)= \Aut(\beta) $ if and only if  $\alpha \sim\beta$. So
the functor $\Aut$   induces an algebraic isomorphism of   groups $\gamma_M \colon \+ B(M) \to \Aut(G)$  for $G=\Aut(M)$.  
  
  \begin{lemma}  \label{lem: top iso}
  	$\gamma_M \colon \+ B(M) \to \Aut(G) $ is   a topological isomorphism. 
  	\end{lemma}
  \begin{proof}
  	  $\+ B(M)$ is a Polish group by \ref{rem:Pol Monoid}, and is algebraically isomorphic to $\Aut(G)$. So, by  \cref{prop:coarsest}(b),  it suffices to show that the action of $\+ B(M)$ on $G$ given by $[\alpha] \cdot g = \gamma_M(\alpha)(g)$ is continuous.  
	     By    \cref{lem: Int,lem: saturation}, the action of $S= \Int(M,M)/\sim$ on $G$ is continuous. So the action of $S \times \check  S $ on $G$  given by $([\alpha]_\sim, [\beta]_\sim)\cdot  g = [\alpha] \cdot  g$ is continuous. This implies the statement. 
  \end{proof}
\subsection{$\+ B(T) $ is topologically isomorphic to $ \Out(G)$}

   \subsubsection*{The Polish  category of $\omega$-categorical theories with interpretations} 	 In the following, all languages are relational, and all theories are assumed to be $\omega$-categorical. We begin by  discussing  {interpretations} of   theories, using a notation that is coherent with  the approach of Ahlbrandt and Ziegler~\cite{Ahlbrandt.Ziegler:86} to the semantic setting. This is somewhat more work than in~\cite{Ahlbrandt.Ziegler:86}, because all the definitions need to be syntactical (though we sometimes use semantic terminology).
 
 \begin{definition}[Interpretation of a theory $U$ in a theory $T$] \label{def:int}
 	An interpretation $\alpha \colon T \rightsquigarrow U$ is given by objects  (a) and  (b):

 	\begin{itemize}
 		\item[(a.i)] a dimension $k \geq 1$ 
 		 [Notation: for each  variable $z$ we have 
 		  a $k$-tuple of variables $\overline{z} = (z_1, \ldots, z_k)$]
 		\item[(a.ii)] formulas $\phi_D(\overline{x})$ and $\phi_E(\overline{x}, \overline{y})$ in the language of $T$ such that \bc  $T \vdash \text{``}D \neq \emptyset\text{''}$, and
 			  $T \vdash \text{``}E \text{ is an equivalence relation on } D\text{''}$. \ec

 	Here, $D$ denotes the definable set (in some model of $T$) given by $\phi_D$, and $E$ the equivalence relation given by $\phi_E$.
 	
		\item[(a.iii)] For each atomic formula $R(x_1, \ldots, x_n)$ in the language of $U$, excluding equality,  a formula $\alpha.R(\overline{x}_1, \ldots, \overline{x}_n)$ in the language of $T$ such that $T$ proves its invariance under the equivalence relation~$E$.
 	 	\end{itemize}
 For any formula $\phi$ in the language of $U$, we define $\alpha.\phi$ to be the formula in the language of $T$ obtained as follows.
 	\begin{itemize}
 		\item[(b.i)] replace each atomic formula $R(y_1, \ldots, y_n)$ with $\alpha.R(\overline{y}_1, \ldots, \overline{y}_n)$,
 		\item[(b.ii)] replace $x = y$ with $\phi_E(\overline{x}, \overline{y})$,
 		\item[(b.iii)] replace each quantifier $\forall x$ (or $\exists x$) with $k$ quantifiers of  the same type over the components of~$\overline{x}$.
 	\end{itemize}
 	
 	We require the following condition to hold:
 	\[
 	C_{T,U}(\alpha): \qquad \text{for every sentence } \psi \in U, \text{ we have } \alpha.\psi \in T.
	\] (End of \cref{def:int}.)
 \end{definition}

Note that if the signature of $U$ is finite,  then the set of  interpretations $T \rightsquigarrow U$ is countable. 
  Interpretations are composed similar to the  case of  structures in \cite[p.\ 65]{Ahlbrandt.Ziegler:86}. The trivial self-interpretation  of $T$ in itself is the interpretation $1_T \colon \,   T \rightsquigarrow T$ where $\phi_D(x)$ is $x= x$, $\phi_E$ is the identity, and $\phi_R$ is  $Rx_1 \ldots x_n$ for each $n$-ary relation symbol~$R$.

  \subsubsection*{The interpretations between theories can be seen as  a path space of a tree}
  \begin{convention} \label{conv: uniq}
  	{\rm We  henceforth assume that for each theory $T$,    the set of formulas in its language is provided with a wellordering of type $\omega$ such as the length-lexicographical ordering,  and that formulas  will be  least in their   class of equivalence under $T$. Then,     for each $n$-tuple of variables $x_1,\ldots, x_n$ there are only finitely many formulas $\phi(x_1, \ldots, x_n)$; recall here that all theories are   assumed to be  $\omega$-categorical.}
  \end{convention}
  
\begin{definition}
	  The  set of interpretations $\alpha \colon T\rightsquigarrow U$ is denoted $\mathrm{Int}(T,U)$.\end{definition} 
	  
	  \begin{remark}
$\mathrm{Int}(T,U)$ can be seen as the set  of  paths of  a subtree $B_{T,U}$ of $\omega^{<\omega} $,  as follows.  Only  the root (the empty string) can be infinitely branching. The  first level of the tree encodes  the triples    $c= \la k, \phi_D, \phi_E\ra$ that $T$ allows. Each further level  (the set of strings of a length $k >1$) is dedicated to a relation symbol $R$ of $L_U$: the extension on that level decides which       formula  of $L_T$ the symbol $Rx_1 \ldots x_n$ is assigned to (this  formula needs to be  invariant under $E$ according to $T$). Let   $B_{T,U}$ be the tree of strings $\sigma$  such that      the  condition  $C_{T,U}(\sigma) $ holds for each sentence $\psi$ such that all   relation symbols of $U$ occuring in it have been assigned by $\sigma$. 
	  \end{remark} 
  
\begin{remark}
	Whether  the  condition  $(*)_{T,U}$ \textit{fails} for an interpretation can be seen from a finite amount of information, so the path space $[B_{T,U}]$ can be identified with the set of interpretations.  We have a natural locally compact, totally disconnected topology on this space where the sub-basic open sets are   of the form $[\sss]$ for $\sss \in  B_{T,U}$. 
\end{remark}

 \subsubsection*{A syntactic version of homotopy, and the quotient category $\mathbb T$}
   Homotopy is simpler for theories than  in the case of models  (\cref{def: homotopy}). We use a    semantic terminology for easier  readability.
  
\begin{definition}[The notation $\approx_T$]  \label{def: homotopy 2} 
  For   interpretations $\alpha_0, \alpha_1 \colon T \rightsquigarrow U$  where $\alpha_i$ is based on (formulas defining) sorts  $D_i/E_i$ in any model, we write   $\alpha_0 \approx_T \alpha_1$  if for some formula  $\phi_\Theta$,  the theory $T$ contains the sentence expressing that 
    $\phi_\Theta$    defines a  bijection $\Theta \colon D_0/E_0 \to D_1/E_1$, and  for each $n$-ary relation symbol $R$ of  $L_U$  the theory $T$ contains the   sentence expressing   that   $\Theta^n (\alpha_0.R ) = \alpha_1.R$.
\end{definition}

  Theories with composition of interpretations form a category. We obtain a quotient category  $\mathbb T$. Its  objects are  the $\omega$-categorical theories in a countable relational language.
       Its  \textit{morphisms} $T \rightsquigarrow U$  are the equivalence classes $[\alpha]_{\approx_T}$ of interpretations. $\mathrm{Mor}(T,U)$ denotes this set of morphisms  from $T$ to $U$ in $\mathbb T$.

\begin{fact}
	Composition of   morphisms in  $\mathbb T$ is   well-defined and associative.  
\end{fact}

\begin{proof}
	Suppose first we have interpretations  $\alpha_0, \alpha_1 \colon  T\rightsquigarrow U$ and $\beta \colon U \rightsquigarrow V$ such that $\alpha_0 \approx_T \alpha_1$ via a formula $\theta$ in the language of $T$. Then   $\beta  \circ \alpha_0   \approx_T \beta  \circ \alpha_1$ via the same formula.
	
	Suppose next  that we have interpretations  $\alpha \colon  T\rightsquigarrow U$ and $\beta_0, \beta_1 \colon U \rightsquigarrow V$ such that $\beta_0 \approx_U \beta_1$ via a formula $\xi$ in the language of $U$. Then $\alpha \circ  \beta_0\approx_T \alpha \circ  \beta_1$ via the formula $\alpha.\xi$.
\end{proof}
 

\begin{definition} \label{def: BT}
	Let   $\+ B(T)$  be the group of  invertible elements of the  topological  monoid $\mathrm{Mor}(T,T)$ (see \cref{rem:Pol Monoid}). 
\end{definition}
 The topology on   the path space $[B_{T,U}]$ induces  a quotient topology on   $\mathrm{Mor}(T,U)$.
\begin{lemma}[Properties of morphism spaces] \label{lem: wawa}  \mbox{ }
	
	{\rm \n  (i) The   projection $q \colon [B_{T,U}] \to  \mathrm{Mor}(T,U)$ that sends $\alpha$ to its equivalence class $[\alpha]_\approx$  is open.  
	
\n 	 (ii) The space $\mathrm{Mor}(T,U)$ has a basis consisting of clopen sets, and hence is totally disconnected.  
	 
	\n    (iii) The space $\mathrm{Mor}(T,U)$  is locally compact.   
	   
	\n    (iv) $\+ B(T)$ is canonically a totally disconnected, locally compact  (t.d.l.c.) group, which is discrete in the case that the signature of $T$ is finite.} \end{lemma}
\begin{proof} We write $\approx$ for $\approx_T$.

	\n  (i) 	  It suffices to  show that  $ [\sss]_{\approx}$ is open in $[B_{T,U}] $  	for each $\sss \in  B_{T,U} -\{\ES\}$.    
	   By $c,d$ etc.\ we denote sorts of $T$, encoded by triples $ \la k, \phi_D, \phi_E\ra$ as in \cref{def:int}. They can be identified with strings on  $ B_{T,U}$ of length $1$. Given $c_i =  \la k_i, \phi_{D_i}, \phi_{E_i}\ra$,   if   $T$ proves  that 
	   $\phi_\Theta$    defines a  bijection $\Theta \colon D_0/E_0 \to D_1/E_1$  such that     for each $n$-ary relation symbol $R$ of  $L_U$ we have  $\Theta^n (\alpha_0.R ) = \alpha_1.R$, then $\phi_\theta$ defines a homeomorphism $h_\theta \colon [c_0]\to [c_1]$.   This implies that $ [\sss]_{\approx}$ is open for each such $\sss$,  as required.

	   \medskip
	   
	 	\n    (ii)    It suffices to  show that  the complement of $ [\sss]_{\approx}$ is open in $[ B_{T,U} ]$  	for each $\sss \in  B_{T,U} -\{\ES\}$.    Suppose the first entry of $\sss$ is $c_0$. Clearly  $\+ D= [c]- [\sss]_{\approx}$ is open. The complement of $ [\sss]_{\approx}$ now consists of all the images  of $\+ D$ under  homeomorphisms $h_\theta  \colon [c_0] \to [c_1] $ as above, together with all $[d]$ such that there is no such homeomorphisms with $[c]$. Hence the complement  of $ [\sss]_{\approx}$ is open as required. 
 
 \medskip
 
	 	\n    (iii)  holds because the open set $[c]_{\approx}$ is compact in  $\mathrm{Mor}(T,U)$ for each sort $c $ of $T$, and each $\alpha_{\approx}$ has such a set as a neighbourhood.

\medskip
	  
	\n  	  (iv) Recall   \cref{rem:Pol Monoid} that the elements of $\+ B(T)$ are represented by pairs $\la [\alpha]_{\approx}, [\beta]_{\approx} \ra \in \mathrm{Mor}(T,T)$ such that $\alpha  \beta  \approx        \beta \alpha  \approx    1_T$. As such they form a closed subset of $\mathrm{Mor}(T,T)$ which is therefore t.d.l.c. Clearly the group operations are continuous. 
\end{proof}

 \subsubsection*{From interpretations of structures to  interpretations of theories}
 \begin{definition}[The forgetful functor $\Th$] \label{def: hatting}
 	Suppose  $T = \Th(M)$ and  $U = \Th(N)$ for countable  structures $M, N$ in relational languages. Let   $\alpha \colon M \rightsquigarrow N$ be a $k$-dimensional  interpretation of structures. Let $\phi_U$ and $\phi_E$ be   formulas defining  in $M$ the domain,   and     $ \{\la \ol x, \ol y \ra \colon \alpha(\ol x) = \alpha(\ol y)\}$, respectively.  Denote by  $\widehat \alpha = \Th(\alpha)$  the interpretation of the theory $U$ in the theory $T$ given by $\la k, \phi_D, \phi_E \ra$  and the assigment  that $\widehat \alpha.R$ is the formula defining  $\alpha^{-1}(R^N)$ in $M$  (which is unique by \cref{conv: uniq}).  
 \end{definition}
 
 If  $\alpha_0, \alpha_1 \in \Int(M,N)$ and $T= \Th(M)$, we will  write $\alpha_0 \approx \alpha_1$ as an abbreviation for $\widehat  \alpha_0 \approx_T \widehat \alpha_1$.  Recall the relation of homotopy on $\Int(M,N)$ from \cref{def: homotopy}.
 \begin{lemma}  \label{lem: same Th}
 	\label{hom versus sim} 
 	Suppose   $\alpha_0, \alpha_1 \in \Int(M,N)$. Then 
 	\bc  $\alpha_0\approx \alpha_1$ iff  there exists some $\pi\in\Aut(N)$ such that $\pi\circ \alpha_0\sim \alpha_1$. \ec
 \end{lemma} 
 \begin{proof} 
 	Let $\alpha_i \colon D_i/E_i \to N$. Recall that the $\alpha_i$  are onto maps.  Suppose the right hand side holds. Clearly  $\pi \circ \alpha_0 \approx \alpha_0$. The hypothesis implies   $\pi\circ \alpha_0\approx \alpha_1$. Since the relation  $\approx $ on $ \Int(M,N)$ is transitive, we conclude that  $\alpha_0\approx \alpha_1$. 
 	
 	Now suppose the left hand side holds  via a formula $\theta$. Let $f\colon D_0/E_0 \to D_1/E_1$ be the bijection defined by $\theta$ in $M$. 
 	Then $\pi=\alpha_1 \circ f \circ \alpha_0^{-1}$ is well-defined and a permutation of $N$. 
 	It suffices to show that $\pi\in\Aut(N)$.  Let $k\ge 1$ and $A\subseteq N^k$ be definable in $N$. 
 	Since $\alpha\approx\beta$ via $\theta$, we have  $f(\alpha_0^{-1}(A))= \alpha_1^{-1}(A)$. Therefore $\pi(A)= \alpha_1(f(\alpha_0^{-1}(A))) = A$. 
 \end{proof}

\begin{lemma}\label{lem:ints theories} \mbox{ }
	{\rm \n  	(i) $\Th$ is a full  functor from the category of $\omega$-categorical countable structures to  the category of $\omega$-categorical theories. 
		
			\n 	(ii) If $\alpha, \beta  \colon M \rightsquigarrow N$ and  $\alpha \sim \beta$, then $\Th(\alpha)\approx \Th(\beta)$. 
			
		\n 	(iii) The functor  $\Th$ induces a surjective group homomorphism $\chi_M \colon \+ B(M) \to \+ B(T)$. 
			
			\n (iv)   $\chi_M$ is continuous.  
			}
\end{lemma}
\begin{proof}   
(i) $\Th$ is clearly a functor. To see that $\Th$ is full, suppose that $\alpha\in \mathrm{Mor}(T,T)$. 
	Let $M$ be a countable model of $T$. By $(*)_{T,U}$, $D/E$ is a model of $T$, where $D$ and $E$ are as in Definition \ref{def:int}. Since $T$ is $\omega$-categorical, there exists an isomorphism $D/E \rightarrow M$. 
	The lifting $\gamma\colon D\rightarrow M$ to $D$ then satisfies $\hat{\gamma}=\alpha$. 
	
	\n (ii) is evident.  
	
	\n 
	 (iii) $\chi_M$ is clearly a homomorphism. To show it is onto, 
  suppose that $(\alpha,\beta)$ represents an element of $\+ B(T)$. 
Since $\Th$ is full by (i), there exist $\gamma, \delta\in \Int(M,M)$ such that  $\hat{\gamma}=\alpha$ and $\hat{\delta}=\beta$. 
Since $\gamma\circ \delta\approx_T \id$,  by Lemma \ref{lem: same Th} there is  $\pi\in\Aut(M)$ such that $(\pi\circ\gamma)\circ \delta\sim \id$. 
Since $\pi\circ \gamma\in \Int(M,M)$ satisfies $\hat{\pi\circ \gamma}=\alpha$, $\chi_M$ is surjective. 
 
 \n  (iv) is clear by (iii) and  since the functor $\Th$ is continuous on $\Int(M,N)$.   
\end{proof}

Recall the topological isomorphism $\gamma_M  $ from \cref{lem: top iso}, and the  surjective  group  homomorphism $\chi_M $ from  (iii) of~\cref{lem:ints theories}:   
\[ \Aut(G) \stackrel{\gamma_M}{\longleftarrow} \+ B(M)  \stackrel{\chi_M}{\longrightarrow} \+ B(T). \]
Our  goal is to show that $\chi_M\circ \gamma_M^{-1}$ induces a  group homeomorphism $\Out(G) \to \+ B (T)$ as required for (ii) of \cref{th:BI Out},
\begin{lemma} \label{lem: triv image}
Let $\alpha \in \Int(M,M)$. Then  
\bc $\chi_M([\alpha]_\sim)= 1 \LR \gamma_M(\alpha)\in \Inn(G)$. \ec
\end{lemma}
\begin{proof}
	First suppose that the right hand side holds. Thus using the notation \cref{def:Autfunctor} there is $g\in G$ such that $\Aut(\alpha) = \Aut(g)$. Therefore $\alpha \sim g$, which implies $\hat \alpha = \hat g$. Clearly $\hat g \approx 1_T$, whence the left hand side holds. 
	
	Now suppose  the left hand side holds. Then $\hat \alpha \approx \hat  {\mathrm {id} _M}$. So  by \cref{hom versus sim}, there is $\pi \in \Aut(M)= G$ such that $\pi \circ \alpha \sim  \mathrm {id} _M$, and hence there is $h= \pi^{-1} \in G$ such that $\alpha \sim h$.  So $\gamma_M(\alpha) = \gamma_M(h)$ which is in $\Inn(G)$. 
\end{proof}


 \begin{proof}[Proof of \cref{th:BI Out}]
 $\+ B(T)$ is totally disconnected, locally compact by  \cref{lem: wawa}. It remains to show that it  is   topologically isomorphic to $\Out(G)$.  
 
 By \cref{lem: top iso} we have a topological isomorphism $\gamma_M^{-1} \colon \Aut(G) \to \+ B(M)$. By \cref{lem: triv image} $\gamma_M^{-1} $ sends $\Inn(G)$ to the kernel $K$ of  $\chi_M$.  So, using that $\chi_M$ is surjective,   $\chi_M\circ \gamma_M^{-1}$ induces a group homeomorphism 
 $ \Out(G) = \Aut(G)/\Inn(G) \cong \+B(M)/K \cong \+ B(T)$, as required. 
\end{proof}

%

\section{Borel equivalence between  \\  categories of groups  and  of meet groupoids}


\subsection{Some category theoretic preliminaries}

\begin{definition}[Mac Lane~\cite{Maclane:98}] \label{def: equivalence categories}  \

	\n (a) Given a category $\+ C$, a  functor $\Upsilon  \colon \+ C \to \+ C $ is \textit{homotopic} to the identity, written $\Upsilon \sim 1_\+ C$, if the following holds. For each object $M$ of $\+ C$,  there is an isomorphism  $\eta_M \colon M \to \Upsilon(M)$ such that the diagram  \[ \xymatrix { X  \ar^p @ {->}  [r]  \ar^{\eta_X} @ {->}  [d] & Y\ar^{\eta_Y} @ {->}  [d] \\ 
		\Upsilon(X)  \ar^{\Upsilon(p)} @ {->}[r]   & \Upsilon(Y)}\]
	commutes for each morphism $p \colon X \to Y$.

\n (b) 	An \textit{equivalence} of categories $\+ C, \+ D$ is given by a pair of functors $\Gamma \colon \+ C \to \+ D$ and $\Delta \colon \+ D \to \+ C$ such that $\Delta \circ \Gamma \sim 1_\+ C$ and $\Gamma \circ \Delta \sim 1_\+ D$. \end{definition} 
	\begin{definition} A \textit{Borel category} is a small  category that can be seen as a Borel structure in the sense of \cite{Montalban.Nies:13}. In particular, the objects and morphisms  form  Borel sets in   appropriate Polish spaces.  If  $\+C$ and $\+ D$ are Borel categories, a Borel equivalence between $\+ C$ and $\+ D$ is an equivalence such that the functors and the assigments $X \to \eta_X$ witnessing the homotopies    are Borel. 
\end{definition}

\subsection{The Borel equivalence}

\begin{definition}[Kechris et al., \cite{Kechris.Nies.etal:18}, in Th.\ 3.1] \label{def:ICB}
	Let $\Gr$ be a Borel class   of closed subgroups of $\S$ closed under conjugation.  We say that $\Gr$ satisfies the \textit{invariant countable  basis condition} (ICB) if one can to $G \in \Gr$ assign in a Borel way a countable set $\+ S_G$ of open  subgroups of $G$ that form a neighbourhood  basis of $1_G$, in    a way that is invariant under topological  isomorphisms of groups in~$\Gr$: if $h \colon G\to H$ is such an isomorphism, then $U \in \+ S_G $ iff $h(U) \in \+  S_H$ for each $U$.  
\end{definition}

\begin{theorem} \label{thm: Borel equivalence} 
	Suppose that a Borel    class $\Gr$ of closed subgroups of $\S$ satisfies the invariant  countable  basis condition. Then $\Gr$ as a category with topological isomorphism is Borel, and there is  a Borel category $\+ M$ of countable structures in a finite signature with isomorphism, and       functors $\+ W \colon \Gr \to \+ M$ and $\+ G \colon \+ M \to \Gr$ that  induce a Borel equivalence of categories. \end{theorem}

As a consequence, $\+ W$  is a full functor. This yields  another (albeit roundabout) proof that $\Aut(G)$, now assuming that  $G \in \Gr$,  can be topologized as a non-Archimedean group:
$\Aut(G)$ is topologically isomorphic to $\Aut(\+ W(G))$. One can easily check that when $\Aut(G)$ carries   this topology, its  action on $G$ is continuous. So by \cref{th:1}(iii) this is  the same topology as the one given in \cref{th:1}(i).

We may   suppose without loss of generality that   for each $G  \in \Gr$ the neighbourhood basis $\+ S_G$ of $1_G$ is closed under finite intersections.  For,  one  can replace the given class of open subgroups $\+ S_G$ by its closure under finite intersections,  maintaining the Borel and   invariance conditions.  We let the domain of      $\+ W(G)$ be the left cosets of  subgroups of $G$  in $\+ S_G$, together with $\ES$. It has a groupoid structure given by product of ``matching" cosets $A,B$ (namely, $A$ is a left coset and $B$ a right coset of the same subgroup), and a lower semilattice structure  given by intersection.   $\+ W(G)$ is called the \textit{meet groupoid} of $G$.

%
%

\cref{thm: Borel equivalence} applies to the class of locally Roelcke precompact, non-Archimedean groups, where  $G$ is locally Roelcke precompact if it has a Roelcke precompact open subgroup (see \cref{def:RP}).  This notion was     introduced in~\cite{Rosendal:21,Zielinski:21} for  the wider context of Polish groups.  We note that 
each t.d.l.c.\  group is   locally Roelcke precompact; in contrast,  the   group  $\Aut(T_\infty)$ of automorphisms of an infinitely branching {unrooted}  tree is locally R.p., without being Roelcke precompact or t.d.l.c.~\cite{Zielinski:21}. 

The following is easily checked.
\begin{prop} \label{cor: lRP}
	The class of locally R.p.\ groups $G$ satisfies the invariant countable basis condition,    taking as   $\+ S_G$    the class of R.p.\ open subgroups of $G$. 
\end{prop}

Throughout this section, let $\Gr$ be a class of closed subgroups of $\S$ that satisfies  the invariant countable basis condition  as in \cref{def:ICB}.   
By $G$ we will always denote a group in $\Gr$.

\subsection{Full meet groupoids}
\begin{definition}
	\label{def groupoid} 
	A \emph{groupoid} consists of a domain $M$, a partial binary operation $\cdot$ and a unary operation ${(.)}^{-1}$ with the following properties for each $A,B,C\in M$: 
	\bi \item[(a)]   $(A \cdot B)\cdot C= 
	A \cdot (B\cdot C)$,  with either both sides or no side defined (and so the parentheses can be omitted in products);  \item[(b)]  $A\cdot A^{-1}$ and $A^{-1}\cdot A$ are always defined; \item[(c)] if $A\cdot B$ is defined then $A\cdot B\cdot B^{-1}=A$ and $A^{-1}\cdot A\cdot  B =B$.\ei 
	
	Given a   groupoid $M$, the letters $A,B, C$ will  range over  elements of $M $, and the letters  $U,V,W$ will range  over idempotents. 
	
\end{definition}
	The following goes back to  \cite{Melnikov.Nies:22} in the context of t.d.l.c.\ groups; also see ~\cite[Section~3]{LogicBlog:22}.

	\begin{definition} \label{def:MeetGroupoid} A   \emph{meet groupoid} is a groupoid  $(M, \cdot , {(.)}^{-1})$ that is also a meet semilattice  $(M, \cap  ,\ES)$ of which  $\ES$ is the  least element.    
		It satisfies the conditions  that  $\ES^{-1} = \ES = \ES \cdot \ES$,   that $\ES \cdot A$ and $A \cdot \ES$ are undefined for each $A \neq \ES$,  and that $U  \cap V \neq \ES$ for idempotents $U,V$ such that $U,V \neq \ES$.
		Further, writing $A \sub B \LR A\cap B = A$, it satisfies 
		
		\bi \item[(d)] $A \sub B \LR A^{-1} \sub B^{-1}$, and
		
		\item[(e)]   if  $A_i\cdot B_i$ are defined ($i= 0,1$) and $A_0 \cap A_1 \neq \ES \neq B_0 \cap B_1$, then \bc $(A_0  \cap A_1)\cdot (B_0 \cap B_1) =  A_0 \cdot  B_0 \cap A_1 \cdot B_1 $ \ec

		%
		\ei
	\end{definition}
\begin{remark} \label{rem: props}
		Since inversion is an order isomorphism,  if  $A \cap B \neq \ES$  then $A^{-1} \cap B^{-1}= (A \cap B)^{-1}$.  	Monotonicity of the groupoid product follows from (e):  
	
	\bi \item	[(f)] if  $A_i\cdot B_i$ are defined ($i= 0,1$) and $A_0 \sub A_1, B_0 \sub B_1 $,
	then  $A_0 \cdot B_0 \sub A_1 \cdot B_1$.  \ei
	Another consequence of (e) is that the intersection of two idempotents is again an idempotent.

\end{remark}

	
	Given meet groupoids $\+ W_0, \+ W_1$, a bijection $h \colon M_0 \to M_1$ is an \emph{isomorphism} if it preserves the three operations.

	\begin{definition} Suppose we are given a groupoid as in  Definition~\ref{def groupoid}. 
		Let $U$ be  an idempotent.  We say $A$ a \textit{left $U$ ${}^*$coset}  if  $A\cdot U=A$. We say $B$ a \textit{right  $U$ ${}^*$coset}  if  $U\cdot B=B$.  
		We write $\LC(U)$ and $\RC(U)$ for the collections of left and right $^*$cosets  \, of $U$, respectively. 
	\end{definition} 
	
\begin{remark} \label{rem:lrcoset inverses}
		Each $A$ is a left $U$ $^*$coset and right $V$ $^*$coset for unique \idempotents\ $U$ and $V$, by cancellation and \ref{def groupoid} (c). Further, $U = A \cdot A^{-1} $ and $V = A^{-1} \cdot A$.
\end{remark}

	\begin{definition} 
		\label{def:CosetGroupoid} 
		A \emph{full meet  groupoid} is a meet  groupoid  $(M, \cdot , {}^{-1}, \cap,\emptyset)$ that additionally satisfies the following conditions for all \idempotents\ $U \sqsubseteq V$: 
		
		
		
		\begin{enumerate} 
			\item[(g)] 
			\label{def:CosetGroupoid bi} 
			(level up) 
			If $A$ is a left (right) $U$ ${}^*$coset,  there exists a unique left (right) $V$ $^*$coset $B$ with $A\sqsubseteq B$. 
			
			\item[(h)] 
			\label{def:CosetGroupoid bii} 
			(level down) 
			Suppose that $B$ is a left (right) $V$ ${}^*$coset. 
			If $U   \sqsubset  V$,    there exist at least two distinct left (right) $U$ $^*$cosets\  $A\sqsubseteq B$. 
			
		\end{enumerate} 
	\end{definition}

	It follows from   \ref{def:CosetGroupoid}(g) that any two distinct left (right) $U$ $^*$cosets\ are disjoint. 
 
 \begin{fact}
{\rm In a full meet groupoid,  if  $A_i$ is a left  $U_i$ ${}^*$coset for $i=0,1$   and $A_0\cap A_1$ is nonempty, then $A_0\cap A_1$ is a left $U_0\cap U_1$  ${}^*$coset. A similar fact holds for right cosets. }
 \end{fact}

	\begin{proof}
Using \ref{def:MeetGroupoid} (e) we have 
	\begin{eqnarray*} U_0 \cap U_1 &=& (A_0^{-1}\cdot A_0) \cap (A_1^{-1}\cdot A_1) \\  &=& (A_0^{-1}\cap A_1^{-1})\cdot (A_0 \cap A_1) \\ 
		&=&  (A_0 \cap A_1)^{-1} \cdot (A_0 \cap A_1). \end{eqnarray*} So $A_0\cap A_1$ is a left $U_0\cap U_1$  ${}^*$coset by \cref{rem:lrcoset inverses}.
	\end{proof}
  In particular, letting $U=U_0= U_1$ and using uniqueness in \ref{def:CosetGroupoid}(g), distinct left (right) ${}^*$cosets of $U$ are disjoint.


\subsection{The functor  $\+ W$}
We  define a functor $\+ W$  from the category  of groups $\Gr $ with topological isomorphisms  to the category of  countable full meet groupoids with isomorphisms. 

\begin{definition}[Functor $\+ W$]\label{def: functorW} 
	Let $\+ W(G)$ be the collection of left cosets of subgroups  in $\+ S_G$, together with $\ES$  (so that it is closed under $\cap$).   Note that $\+ W(G)$ is a basis that  is  closed under inverse because $\+ S_G$ is closed under conjugation.  We write $A \cdot B = C$ if $A$ is left coset and $B$ is right coset of the same  subgroup, and $AB = C$. Clearly $\+ W(G)$ is closed under this operation. So  $\+ W(G) $ forms a groupoid.    We extend the groupoid operation to $\+ W(G)$ by letting $\ES \cdot \ES = \ES$, and $\ES \cdot A$ is undefined for $A \neq \ES$.  The structure  $(\+ W(G), \cdot, \cap) $ is called the  \emph{meet groupoid} of $G$. 
	If $\alpha \colon G \to H$ is a topological  isomorphism, we define $\+ W(\alpha)(A) = \{\alpha(g) \colon \, g \in A\}$ for any $A \in \+ W(G)$. 
\end{definition}

We  show that $\+ W(G)$ is indeed a full meet groupoid. Clearly \ref{def:MeetGroupoid}(d) and monotonicity \ref{def:MeetGroupoid}(f) hold.  The    condition to check   is \ref{def:MeetGroupoid}(e).  Let $A_i$ be a right coset of a subgroup $U_i$ and a left coset of subgroup $V_i$,  so that  $B_i$ is  a right coset of $V_i$ by hypothesis. Then  $A_0 \cap A_1 $ is a  right  coset of $U_0 \cap U_1$, and $B_0 \cap B_1$ a right coset of $V_0 \cap V_1$, so the left hand side in (e) is defined.   Clearly the left side is contained in the right hand side by monotonicity. The right hand side is also a right coset of $U_0 \cap U_1$, so they are   equal.

\begin{definition}[Full filters, \cite{Nies.Schlicht.etal:21}, Def.\ 2.4] \label{def:FF}
	A \textit{filter}  $R $ on  the partial  order $ (\+ W(G) , \sub)$ is  a subset that is directed downward and closed upward. It  is called \emph{full}  if each   subgroup  in  $\+ W(G)$ has a  left and a right coset in $R$.  (These cosets are necessarily unique.)
\end{definition}

\begin{lemma}  \label{lem:KechrisNies}
	There is a canonical bijection between the elements of  $G$ and   the set of full filters on $  \+ W(G)$. It is given by \bc $g \mapsto R_g :=  \{ gU \colon \, U \in \+ S_G\}$. \ec It inverse is given by $R \mapsto  g$ where $\{ g\}= \bigcap R$.  
\end{lemma}
\n  To verify this, first note that $R_g$ is indeed a full filter: given a subgroup $V \in \+ S_G$, let $U = g^{-1} Vg$; then $U \in \+ S_G$ by invariance. So $Vg = gU $ is a right coset of $V$ in  $R_g$. 
The main point is to show that   $\bigcap R$ is non-empty for each full filter~$R$.  This  is proved similar to  \cite[Claims 3.6 and 3.7]{Kechris.Nies.etal:18}; also see the proof of \cite[Prop.\ 2.13]{Nies.Schlicht.etal:21}. It is then easy to see that $\bigcap R$ is a singleton, using that $\+ B$ is a basis for the topology on $G$.

\subsection{The functor $\+ G$  on the category of full meet groupoids}
We next define a  functor  $\+ G$  from the category of countable full meet groupoids   (\cref{def:CosetGroupoid}) to the   category of non-Archimedean groups with topological isomorphism. Later on, we will restrict it to the category $\+ M$. We actually need the definition of  the functor~$\+ G$ on the larger category to define the Borel category $\+ M$: one of the conditions on a full meet groupoid $M$ for being an object  of  $\+M$  is that $\+ G(M) \in \Gr$, which is a Borel condition. 
\begin{definition}[Functor $\+ G$] \label{df: inverse functor}
	For a countable full meet groupoid $M$, let

	$\+ G(M)=\{ p \in \Sym(M) \colon \, p \text{ is a  $(M,\cap) $ automorphism } \land $

	\hfill  $ p(A \cdot B) = p(A ) \cdot B \text{ whenever }   A \cdot B \text{ is defined} \}$. 
	
\n 	If $\theta \colon M \to N$ is an isomorphism of meet groupoids, let $\+ G(\theta)(p)= \theta \circ p \circ \theta^{-1}$.  
\end{definition}

\begin{remark} \label{rem:expl G(M)}
	 Clearly, when $\theta \colon M \to N$ is an isomorphism then  $\+G(\theta)$ is an isomorphism $\+ G(M) \to \+G(N)$, with inverse  $\+G(\theta^{-1})$. 
	 
 $\+ G(M)$ is the set of automorphisms of the structure $M'$ with the binary function  $\cap$ and unary partial functions $f_B$ for each $B \in M$,  that send $A $ to $ A \cdot B$ when defined. It is thus a closed subgroup of   $\Sym(M)$.  
\end{remark}

\begin{remark} The elements of $\+G(M)$ correspond  to full filters (\cref{def:FF}),  as follows: \bi \item [(a)] Given $p\in  \+G(M)$ let $R = \{ pU \colon U\in \+ N_G\}$. \item[(b)] Given a full filter $R$, let $p(U) = A$ where $A$ is the left coset of $U$ in $R$; this determines the values on all cosets because $p(B) = p(U)\cdot B$ when $B$ is a right coset of $U$. \ei 
 \end{remark}

\subsection{Towards the homotopies of categories}
We aim to show that after suitably refining the category of full meet groupoids,  the functors $\+ W$ and $\+ G$ induce an equivalence of categories according to~\cref{def: equivalence categories}. For this we need to define two  homotopies: 
\begin{definition}\label{def:etaM}
	Given  a full meet groupoid $M$, define a  map $\eta_M  \colon M \to \+ W(\+ G(M))$ as follows. For  $A \in M - \{ \ES\}$ 
	let \bc $\eta_M(A):=\{ p\in \+ G(M)\colon  p(U)=A\}$ \ec  where $A$ is a left $U$ $^*$coset. Also  let $\eta_M(\emptyset):=\emptyset$.   \textit{ If $M$ is understood from the context we also write  $ \hat A$ for  $\eta_M(A)$. }
\end{definition}

The finite intersections of sets $\hat A$ form a basis for the topology  of $\+ G(M)$. This uses that for $p \in \+ G(M)$ and $C,D \in M$,  $p (C)= D$ is equivalent to $p(C)\cdot C ^{-1}= D\cdot C^{-1}$, or again $p(U) = A $ where $U = C\cdot C^{-1}$ and $A  = D \cdot C^{-1}$.

\begin{definition} \label{def: etaG} 
For $G \in \Gr$ define a map  $\eta_G \colon G \to \+ G(\+ W(G))$ 
	by  $\eta_G(g)(A)= gA$.
\end{definition}
\begin{lemma} \label{lem:eta isom}
	$\eta_G \colon G \to \+ G(\+ W(G))$ is a  topological isomorphism. 
\end{lemma}
\begin{proof}
	Clearly $\eta_G$  is a group monomorphism.  To show that it is onto, given $p \in \+ G(\+ W(G))$ let $R = \{ pU \colon  U \in \+ S_G\}$. Since $R$ is a full filter on $\+ W(G)$,   by \cref{lem:KechrisNies} there is a unique $g \in G $ such that $\bigcap R= \{ g\}$. Clearly $p = \eta_G(g)$. 
	
	To check $\eta_G$ is continuous, take a sub-basic subset  $\hat A$  of $\+ G(\+ W(G))$,  where $A \in \+W(G)$. Note that  $\eta_G^{-1}(\hat A) = A$ is open. 
	
	Given that both $G$ and $\+ G(\+ W(G))$ are Polish, this shows $\eta_G$ is a topological isomorphism. 
\end{proof}


The next lemmas will  show that the map  $\eta_M$ from \cref{def:etaM}  is an embedding  of meet groupoids.

Note that $\hat{U}$ is an open subgroup of $\+ G(M)$. 
Furthermore, for any left $U$ $^*$coset $A$, the set  $\hat{A}$ is a left $\hat{U}$ coset,  because $x\hat{U}=\hat{A}$ for any $x\in \hat{A}$. 
Similarly, $\hat{B}$ is a right $\hat{V}$ coset for every right $V$ $^*$coset $B$, since $x(B)=x(V\cdot B)=x(V)\cdot B=V\cdot B=B$ for any $x\in \hat{V}$ and hence $\hat{V}y=\hat{B}$ for any $y\in \hat{B}$.

We first  show that the  map $\eta_M$ sends $^*$cosets to open cosets of $\+ G(M)$. 
\begin{lemma} 
	\label{A hat nonempty} 
	$\hat{A}\neq\emptyset$ for any   $^*$coset $A \neq \ES$. 
\end{lemma} 

\begin{proof} 
	First suppose that $M$ has a least idempotent $U\neq \ES$. 
	Suppose that $A$ is a left $V$ $^*$ coset. 
	Then $A$ contains a $U$ $^*$coset $B$ by \ref{def:CosetGroupoid}(h). 
	We define $f\in \G(M)$ as follows. 
	Suppose that $W$ is any \idempotent\ with $U\sqsubseteq W$. 
	There is a unique left $W$ $^*$coset $C$ with $A\sqsubseteq C$ by \ref{def:CosetGroupoid} (g). 
	Let $f(W)=C$ and $f(B)=C\cdot B$ for any right $W$ $^*$coset $B$. 
	Then $f\in \hat{B}\subseteq \hat{A}$. 
	
	Now suppose $M$ has   no least nonempty idempotent.  
	Suppose that $A$ is a left $U$ $^*$coset.  Let  $\langle U_n \mid n\in\NN\rangle$ enumerate all   \idempotents\ of $M$, such that $U_0=U$. 
	We construct a strictly increasing sequence $\langle n_i \mid i\in\NN\rangle$ and an increasing sequence $\langle f_i \mid i\in\NN\rangle$ of partial functions on $M$ whose union will be an element of $\hat{A}$.

	Let $n_0=0$ and suppose that $V$ is any \idempotent\ with $U_0\sqsubseteq V$. 
	There is a unique left $V$ $^*$coset $A$ with $A_0\sqsubseteq A$ by \ref{def:CosetGroupoid} (g). 
	Let    $f_0(B)=A \cdot B$ for any right $V$ $^*$coset $B$; in particular, $f_0(V)=A$. 
	Note that $f_0(B)$ is a right $A V A^{-1}$ coset and $f_0(VA^{-1})=A V A^{-1}$.

	If $f_i$ has been defined, let $n_{i+1}>n_i$ be least such that \bc $U_{n_{i+1}}\subseteq U_{n_i}\cap (A_{n_i} U_{n_i} A_{n_i}^{-1}) \cap U_j$ for all $j<n$.  \ec
	Note that $n_{i+1}$ exists since $H$ is not discrete. 
	There exists a left $U_{n_{i+1}}$ $^*$coset $A_{n_{i+1}}\sqsubseteq f(U_{n_i})$ by \ref{def:CosetGroupoid} (h). 
	Suppose that $V$ is any \idempotent\ with $U_{n_{i+1}}\sqsubseteq V$. 
	Let $f_{i+1}(V)$ be the unique left $V$ $^*$coset $A$ with $A_{n_{i+1}}\sqsubseteq A$ by \ref{def:CosetGroupoid}(g). 
	Let  $f_{i+1}(B)=A \cdot  B$ for any right $V$ $^*$coset $B$. 
	Note that $f_{i+1}(B)$ is a right $A V A^{-1}$ coset and $f_{i+1}(VA^{-1})=A V A^{-1}$. 
	It follows that every right $^*$coset of $U_{i+1}$ is in the domain and every right $^*$coset of $U_i$ in the range of $f_{n_{i+1}}$. 
	Since $f_i\subseteq f_{i+1}$ by construction and each $f_i$ is a partial automorphism of $M'$  in the sense of \cref{rem:expl G(M)}, their union is an automorphism of $M'$  and hence in $\hat{A}$. 
\end{proof}

We next  show that the  map $\eta_M$  is compatible with intersections, products and inverses. 
Let $\hat{A}\hat{B}$ denote the setwise product of $\hat{A}$ and $\hat{B}$ and $\hat{A}^{-1}$ the setwise inverse of $\hat{A}$ in $\G(M)$. 
Note that for any \idempotent\ $U$, $\hat{U}$ is an open subgroup of $\G(M)$.


\begin{lemma}
	\label{subset is correct} 
	The following hold for all $^*$cosets\ $A$ and $B$. 
	\begin{enumerate-(a)} 
		\item
		\label{subset is correct 1} 
		$\hat{A\wedge B}=\hat{A}\cap \hat{B}$. 
		
		\item 
		\label{subset is correct 2} 
		$\hat{A\cdot B}=\hat{A} \hat{B}$  if $A\cdot B$ is defined. 
		
		\item 
		\label{subset is correct 4} 
		$\hat {A^{-1}} = \hat{A}^{-1}$. 
		%
		%
	\end{enumerate-(a)} 
\end{lemma} 
\begin{proof} 
	\ref{subset is correct 1} 
	Suppose that $A$ is a left $U$ $^*$coset and $B$ is a left $V$ ${}^*$coset. 
	If $A\wedge B= \ES$, then $\hat A \cap \hat B= \ES$. 
	To see this, note that for any $x\in \hat{A}$ and $y\in \hat{B}$, we have $x(U\wedge V)\neq y(U\wedge V)$, since $x$ and $y$ preserve order. 
	Now suppose that $A\wedge B\neq  \ES$. 
	We argued after Definition \ref{def:CosetGroupoid} that $A\wedge B$ is a left $U\wedge V$ ${}^*$coset. 
	By uniqueness in \ref{def:CosetGroupoid} (g), every element of $\hat{A\wedge B}$ is an element of both $\hat{A}$ and $\hat{B}$. 
	The converse holds, since every element of $\G(M)$ is an automorphism  of the structure  $(M, \wedge)$. 
	
	\ref{subset is correct 2} 
	Suppose that $z\in \hat{A\cdot B}$, where $A$ and $B$ are as above. 
	Then $z(V)=A\cdot B$. 
	Pick any $x\in \hat{A}$ and let $y:=x^{-1}z$. 
	Then $y(V)=x^{-1}z(V)=x^{-1}(A\cdot B)=U\cdot B=B$. 
	Conversely, suppose that $x\in \hat{A}$ and $y\in \hat{B}$. 
	Then $xy(V) = x(B) = x(U\cdot B) =  x(U)\cdot B = A\cdot B$ and hence $xy\in \hat{A\cdot B}$. 
	
	\ref{subset is correct 4} 
	Note that $\hat{A}$ is a left $\hat{U}$ coset, since $x\hat{U}=\hat{A}$ for any $x\in \hat{A}$. 
	Furthermore, $A^{-1}$ is a right $U$ $^*$coset by \ref{def groupoid}, so $\hat{A^{-1}}$ is a right $\hat{U}$ $^*$coset. 
	The claim holds since $\hat{A^{-1}} \cdot \hat{A}=\hat{A^{-1} \cdot A} =\hat{U}$ by \ref{subset is correct 2}. 
\end{proof}

\begin{lemma}\label{lem:eta inj} 
	The map $\eta_M$ is injective and preserves the  order  in both directions.  
\end{lemma}
\begin{proof}
	By   Lemma \ref{subset is correct} \ref{subset is correct 1} it suffices to show that the map is injective. Suppose that $A\neq B$ 
	and assume that $A\not\sqsubseteq B$, so that $A\wedge B \neq A$. 
	If $A\wedge B$ is a $U$ ${}^*$coset then by \ref{def:CosetGroupoid}(h), $A$ contains a left $U$ $^*$coset $D$ disjoint from $A\wedge B$ and thus from $B$. 
	Then $\hat{D}\cap \hat{B}=\emptyset$ and $\hat{D}\subseteq \hat{A}$ by \ref{def:CosetGroupoid} (g), so that $\hat{A}\neq\hat{B}$. 
\end{proof}
\begin{lemma}\ 
	\label{correct form of left cosets} 
	\begin{enumerate-(a)} 
		\item 
		\label{correct form of left cosets 1} 
		For any left (right) $U$ $^*$coset $A$, $\hat{A}$ is a left (right) $\hat{U}$ coset in $\G(M)$. 
		
		\item 
		\label{correct form of left cosets 2} 
		For any left $\hat{U}$ coset $g\hat{U}$ in $\G(M)$, there exists a left $U$ $^*$coset $A$ such that $\hat{A}=g\hat{U}$. 
		A similar fact holds for right cosets. 
	\end{enumerate-(a)} 
\end{lemma} 
\begin{proof} 
	\ref{correct form of left cosets 1} This was shown in the proof of Lemma \ref{subset is correct} \ref{subset is correct 4}. 
	
	\ref{correct form of left cosets 2} 
	Suppose that $f\in\hat{A}$. 
	We claim that $f\hat{V}= \hat{A}$. 
	We have $f\hat{V}\subseteq \hat{A}\hat{V}=\hat{A}$, since $\hat{A}$ is a left coset of $\hat{V}$ by \ref{correct form of left cosets 1}. 
	To see that $\hat{A}\subseteq f\hat{V}$, let $g\in \hat{A}$. 
	Since $f,g\in \hat{A}$, $f^{-1}g\in \hat{A}^{-1} \hat A=\hat{A^{-1}} \hat A\sqsubseteq \hat{V}$ by Lemma \ref{subset is correct} \ref{subset is correct 2} and \ref{subset is correct 4}. 
	Hence $g \in f\hat{V}$. 
\end{proof} 
In particular, a $^*$coset $A$ is a left $U$ $^*$coset if and only if $\hat{A}$ is a left $\hat{U}$ coset. 
A similar fact holds for right $^*$cosets.

\subsection{The category $\+ M$}

\begin{definition} \label{def:MM}
Let  $\+ M$ be the category that has as  objects    the full meet groupoids $M$ with domain $\omega$ such that  \bc $\+ G(M)\in \Gr$ and 
	$\forall \+ U \in \+ S_{\+ G(M)} \exists U \in M \   [ \hat U = \+ U]$. \ec  As before, the morphisms are the isomorphisms of meet groupoids.   
\end{definition}

\begin{lemma}
	For an object $M$ of $\+ M$, the map $\eta_M$ is an isomorphism of meet groupoids.
\end{lemma}
\begin{proof}
	  \cref{lem:eta inj}   showed that $\eta_M$ is an embedding for each full meet groupoid $M$. Now   by the definition of the objects of $\+ M$,  \cref{correct form of left cosets}\ref{correct form of left cosets 2}  shows that $\eta_M$ is onto when $M $ is in $\+ M$.
\end{proof}
\begin{lemma} \label{lem:Borel hat}
	The operation  sending a full meet groupoid $M$ with domain $\omega$ to the map $A \mapsto \hat A$ (as an element of  $\+ F(\S)^\omega$) is Borel.
\end{lemma}
\begin{proof}
	It suffices to show that the map sending a pair $(M, A)$,  where $A \in M$, to $\hat A$ has  Borel graph. This is evident because $A$ is a left coset of $U= A^{-1} \cdot A$, and $p \in \hat A$ iff $p(U)= A$.
\end{proof}
By the hypothesis on the class $\Gr$ and \cref{lem:Borel hat}, the objects of  $\+ M$ form  a Borel set.   So $\+ M$ is a Borel category.

\subsection{Proof of \cref{thm: Borel equivalence} } 
\begin{lemma}
	If $G \in \Gr$ then $\+ W(G) $ is an object of $ \+ M$. 
\end{lemma} 

\begin{proof}
	Let $M = \+ W(G)$. For the first condition in \cref{def:MM}, note that we have $\+ G(M) \in \Gr $ by \cref{lem:eta isom} and since the class $\Gr$ is  isomorphism invariant. For the second condition, let $\+ U \in \+ S_{\+ G(M)}$. Since $\eta_G \colon G \to \+ G(M)$ is a topological  isomorphism and the assignment of  $H \in \Gr$ to $\+ S_H$ is isomorphism  invariant, we have $U: = \eta_G^{-1}(\+ U) \in \+ S_G$, so that $U \in M$. Then $\hat U = \+ U$ as required. 
\end{proof}

Let $\+ G'$ be the functor $\+ G$   restricted to the Borel category $\+ M$. 
We  will show that the diagrams in \cref{def: equivalence categories} are commutative for the functors $\Upsilon = \+ G' \circ \+ W$ and $\Upsilon = \+ W \circ \+ G'$ using the assignments $G \to \eta_G$, and $M \to \eta_M$, respectively, and that the maps $\eta_G$ and $\eta_M$ are isomorphisms  in their respective categories $\Gr$ and $\+ M'$  (recall here that all morphisms in the categories we consider  are  isomorphisms). 

\begin{lemma} (a) $\+ G ' \circ \+ W \sim 1_{\Gr}$ via $G \mapsto \eta_G$

	(b) $\+ W \circ \+ G' \sim 1_{\+M }$ via  $M \mapsto \eta_M$. \end{lemma}
\begin{proof}
	(a) Write $\Upsilon = \+ G' \circ \+ W$. Clearly $\eta_G \colon G \to \Upsilon(G)$ is in the category $\Gr$.  Let $\alpha \colon G \to H$ be a morphism in $\Gr$.  Note that by \cref{df: inverse functor} we have $\Upsilon(\alpha)(p)= \+ W(\alpha) \circ p \circ \+ W(\alpha)^{-1}$ (recall that $\+ W(\alpha) $ is given by  the action of $\alpha$ on open cosets of $G$).  We need to  verify the commutativity of the diagram    in \cref{def: equivalence categories}, which here is
	
	  \[ \xymatrix { G  \ar^\alpha @ {->}  [r]  \ar^{\eta_G} @ {->}  [d] & H\ar^{\eta_H} @ {->}  [d] \\ 
		\Upsilon(G)  \ar^{\Upsilon(\alpha)} @ {->}[r]   & \Upsilon(H)}.\]
We need to show,   for each $g\in G$, the equality of two maps on $\+ W(H)$ depending on~$g$:
	\[ \Upsilon(\alpha)(\eta_G(g))= \eta_H(\alpha(g)).\]
	Applying the map on the right hand side  to $C \in \+ W(H)$ yields $\alpha(g)C$ by Definition \ref{def: etaG} of $\eta_H$. 
	Applying the map on the left hand side yields $\+ W(\alpha)( g ( \+ W(\alpha)^{-1}(C))= \alpha(g \alpha^{-1}(C)) = \alpha(g)C$ by Definition \ref{def: functorW} of the functor $\+ W$. 
	
\n 	(b) Now write  $\Upsilon = \+ W \circ \+ G'$. By the second condition in \cref{def:MM},  $\eta_M \colon M \to \Upsilon(M)$ in the category $\+ M$. Let $\theta \colon M \to N$ be a morphism in $\+ M$.  
	For commutativity of the diagram
	  \[ \xymatrix { M  \ar^\theta @ {->}  [r]  \ar^{\eta_M} @ {->}  [d] & N\ar^{\eta_N} @ {->}  [d] \\ 
		\Upsilon(M)  \ar^{\Upsilon(\theta)} @ {->}[r]   & \Upsilon(N)}\] 
	we need, for each $A \in M$, the equality of two open subsets of $\Upsilon(N)$:   \[ \eta_N(\theta(A))= \Upsilon(\theta)(\hat A).\]

	Suppose $A$ is a left ${ }^*$coset of $U$. Let $q \in \Upsilon(N)$. Note that  $q$ is in the left hand side if $q(\theta(U)) = \theta(A)$ by Definition \ref{def:etaM}. 
	For the right hand side, write $\hat A=\eta_M(A)$. Note that by \cref{df: inverse functor} we have
	\[ \Upsilon(\theta)(\hat A)= \{ \+ G'(\theta) (p) \colon p \in \hat A\}= \{ \theta \circ p \circ \theta^{-1} \colon p \in \hat A\}.\] 
	Therefore, $q$ is in the right hand side if $p(U) = A$ where $p \in \Upsilon(M)$ is given as  $p= \theta^{-1} \circ q \circ \theta$. 
	  Since $\theta$ is an isomorphism, this also says that $q(\theta(U)) = \theta(A)$. So the two sides are equal as sets.
\end{proof}

	
%

 \def\cprime{$'$} \def\cprime{$'$}

\end{document}